\documentclass[12pt,reqno,a4paper]{amsart}
\usepackage[utf8]{inputenc}
\usepackage[T1]{fontenc}
\usepackage{amssymb}
\usepackage{pgf,tikz,pgfplots}
\pgfplotsset{compat=1.12}
\usepackage{mathrsfs}
\usetikzlibrary{arrows}
\usepackage{bm}
\usepackage[alpine]{ifsym}
\usepackage{xpicture}
\usepackage{calculator}
\usepackage{graphicx}

\usepackage[a4paper,top=3.3cm,bottom=3cm,left=3cm,right=3cm,bindingoffset=5mm]{geometry}

\def\sideremark#1{\ifvmode\leavevmode\fi\vadjust{\vbox to0pt{\vss 
      \hbox to 0pt{\hskip\hsize\hskip1em           
 \vbox{\hsize2cm\tiny\raggedright\pretolerance10000
 \noindent #1\hfill}\hss}\vbox to8pt{\vfil}\vss}}} %
\usepackage{color}

\usepackage[colorlinks=true]{hyperref}
\newtheorem{theorem}{Theorem}[section]
\newtheorem{lemma}[theorem]{Lemma}

\newtheorem{proposition}[theorem]{Proposition}

\theoremstyle{definition}
\newtheorem{example}[theorem]{Example}

\newtheorem{remark}[theorem]{Remark}
\newtheorem{definition}[theorem]{Definition}

\numberwithin{equation}{section}

\begin{document}
\title[Discrete equivalence of non-positive at infinity plane valuations]{Discrete equivalence of non-positive at infinity plane valuations}

\author[C. Galindo]{Carlos Galindo}

\address{Universitat Jaume I, Campus de Riu Sec, Departamento de Matem\'aticas \& Institut Universitari de Matem\`atiques i Aplicacions de Castell\'o, 12071
Caste\-ll\'on de la Plana, Spain.}\email{galindo@uji.es}  \email{cavila@uji.es}

\author[F. Monserrat]{Francisco Monserrat}
\address{Instituto Universitario de
Matem\'atica Pura y Aplicada, Universidad Polit\'ecnica de Valencia,
Camino de Vera s/n, 46022 Valencia (Spain).}
\email{framonde@mat.upv.es}

\author[C.J. Moreno-\'Avila]{Carlos Jes\'us Moreno-\'Avila}

\thanks{Partially supported by the Spanish Government, grants  MTM2015-65764-C3-2-P, MTM2016-81735-REDT, PGC2018-096446-B-C22 and BES-2016-076314, as well as by Generalitat Valenciana, grant AICO-2019-2023 and Universitat Jaume I, grant UJI-B2018-10.}

\begin{abstract}
Non-positive at infinity valuations are a class of real plane  valuations which have a nice geometrical behavior. They are divided in three types. We study the dual graphs of non-positive at infinity valuations and give an algorithm for obtaining them. Moreover we compare these graphs attending the type of their corresponding valuation.
\end{abstract}

\maketitle

\section{Introduction}
The concept of non-positive at infinity (NPI) valuation was introduced in \cite{GalMon} (see also, \cite{Mond}) and afterwards extended in \cite{GalMonMor}. Valuations of fraction fields of local regular rings centered at them and their graded algebras are important tools for studying local uniformization of singularities \cite{Zar3,Abh1,ZarSam,Tei,Tei2}. Some recent advances in the topic are \cite{Herol,Cut2,Nov}. NPI valuations are only defined for the bi-dimensional case and have the advantage that they provide not only local but also global information about the surfaces where they are considered.

Valuations of fraction fields of  bi-dimensional regular local rings $R$ centered at them are named, in this paper, plane valuations. They are in one-to-one correspondence with the set of simple sequences of point blowing-ups starting at Spec$R$. These valuations were classified in five types  by Spivakosky \cite{Spiv} attending the dual graphs attached to the above mentioned sequences. Divisorial and irrational valuations (in the terminology of \cite{FavJon1}) are two of these types. They are valuations whose semigroup of values is included in $\mathbb{R}$.

NPI valuations $\nu$ are regarded over the field of rational functions of the projective plane $S=\mathbb{P}^2$ or of some Hirzebruch surface $S=\mathbb{F}_\delta$. They are centered at some point $p \in S$ and must satisfy that $\nu(f) \leq 0$ for all $f \in \mathcal{O}_S (S \setminus T)$, $T$ being a suitable fixed curve on $S$ considered as at the infinity (see Definition \ref{def32}). In the projective case, NPI valuations are centered at infinity, a class of valuations considered by several authors \cite{CamPilReg,FavJon2,FavJon3,Jon,Mond}.

NPI valuations $\nu$ make possible to explicitly determine some global and local objects of the geometry linked to $\nu$ for which there is no explicit description in the general case. For instance, the Seshadri-related  constants introduced in \cite{CutEinLaz} (see \cite{DumHarKurRoeSze} for the plane case) can be explicitly described for NPI valuations \cite{GalMon,GalMonMor2} and are essential for providing some evidences to the valuative Nagata conjecture (which implies the Nagata classical conjecture) \cite{GalMonMoy} (see also \cite{DumHarKurRoeSze}). Newton-Okounkov bodies \cite{Oko3,KavKho,CilFarKurLozRoeShr} of flags determined by valuations of this type can also be explicitly described \cite{GalMonMoyNic2,GalMonMor2}. Finally, NPI divisorial valuations characterize those rational surfaces $Z$ (given by simple sequences of point blowing-ups) with polyhedral cone of curves and minimum number of extremal rays.  Even, in the projective case, they allow us to decide when the Cox ring of $Z$ is finitely generated \cite{GalMon,GalMonMor}.

Although a valuation is an algebraic object, the concept of NPI valuation depends on the (classical) minimal surface where it is considered giving rise to three types of valuations. We consider projective and Hirzebruch (real plane) valuations and divide the last ones in two sets, special and non-special valuations (see Definition \ref{def31}). The definition of NPI valuation varies slightly in the above three cases according the corresponding curve $T$ at infinity (see Definition \ref{def32}).

We think it is interesting to compare these types of valuations for a better knowledge of the surfaces that provide and, also, in the hope of finding a more purely algebraic definition.

In this paper we focus on dual graphs of NPI valuations comparing those corresponding to valuations of the above mentioned types. Our main result, Theorem \ref{Thm_Comparacion_grafos}, shows the inclusions among the sets of dual graphs admitting some valuation of the above three types. We also prove that these inclusions are strict. For simplicity of our exposition, dual graphs are represented by tuples of rational or real numbers and valuations having the same tuple are called discretely equivalent (see Definition \ref{Def_dicrete_class}). Within the sets of projective, special Hirzebruch and non-special Hirzebruch valuations,  the discrete classes corresponding to NPI valuations are simply determined by an inequality (see Theorem \ref{Th_CondNumeric_divisorial}). This means that a valuation  whose discrete class does not satisfy one of the above mentioned inequalities is not NPI of the corresponding type.

Geometrically speaking, what we prove is that the number of extremal rays of the cone of curves of a rational surface (coming from a simple sequence $\pi$ of point blowing-ups) is not the minimum possible whenever the corresponding inequality in Theorem \ref{Th_CondNumeric_divisorial} (whose data are only determined by the dual graph of $\pi$) does not hold. Otherwise, to decide if the cone of curves has (or not) the simplest shape we need to know the set of initial free points in $\pi$.

Generalities about plane valuations are given in Section \ref{Sect_Uno} and NPI valuations are introduced in Section \ref{Sect_Dos} where we show some of their properties. The core of the paper is Section \ref{Sect_Tres}, here we make a comparative study of the discrete equivalence classes corresponding to  NPI valuations. Finally, Section \ref{Sect_Cuatro} provides an algorithm which inductively determines dual graphs of NPI valuations. The algorithm can  potentially give all dual graphs corresponding to NPI valuations.

\section{Plane valuations}\label{Sect_Uno}

\subsection{Generalities}

A valuation $\nu$ of a field $K$ is an onto map $$\nu:K^*(=K\setminus\{0\})\to G,$$ $G$ being a commutative group, which for $a,b\in K^*$ satisfies
$$
\nu(a+b)\geq \min\{\nu(a),\nu(b)\} \ \ \text{and} \ \ \nu(ab)=\nu(a)+\nu(b).
$$
The \emph{valuation ring} of $\nu$, $R_\nu:=\{f\in K^*\ | \ \nu(f)\geq 0\}\cup \{0\}$, is a local ring whose maximal ideal is $\mathfrak{m}_\nu=\{f\in K^* \ | \ \nu(f)>0\} \cup \{0\}$. We say that $\nu$ is \emph{plane} whenever $K$ is the fraction field of a bi-dimensional regular local ring $(R,\mathfrak{m})$ and it is \emph{centered at }$R$ which means that $R\cap\mathfrak{m}_\nu=\mathfrak{m}$. We also assume that $k:=R/\mathfrak{m}$ is algebraically closed. According \cite{Spiv}, plane valuations correspond one-to-one to simple sequences of point blowing-ups:
\begin{equation}\label{Eq_sequencepointblowingups}
\pi: \cdots \rightarrow Z_n\xrightarrow{\pi_n} Z_{n-1}\rightarrow \cdots \rightarrow Z_1 \xrightarrow{\pi_1} Z_0=\text{Spec}R,
\end{equation}
where $\pi_1$ is the blowing-up of Spec$R$ at the closed point $p=p_1$ defined by $\mathfrak{m}$ and $\pi_{i+1}$, $i \geq 1$, the blowing-up of  $Z_i$ at the unique point $p_{i+1}$ in the exceptional divisor $E_i$ created by $\pi_i$ satisfying that $\nu$ is centered at the local ring $\mathcal{O}_{Z_i,p_{i+1}}$.

The set of points $\mathcal{C}_\nu=\{p=p_1,p_2,\ldots\}$ is the \emph{configuration of infinitely near points} of $\nu$. We say that $p_i$ is \emph{proximate} to $p_j$, denoted by $p_i\to p_j$, if $i>j$ and $p_i$ belongs to the strict transform of $E_j$ on $Z_{i-1}$. Strict transforms of divisors $E_i$ are also denoted by $E_i$. Points $p_i$ (and  exceptional divisors $E_i$) such that $p_i$ is proximate to other point $p_j$ such that $j<i-1$ are called \emph{satellite}; otherwise they are named \emph{free}.

\subsection{The dual graph}
The dual graph $\Gamma_\nu$ of a plane valuation is a (finite or infinite) labelled tree whose vertices represent the exceptional divisors $E_i$ (they are labelled with the number $i$) and two vertices are joined by an edge whenever $E_i \cap E_j \not = \emptyset$. Spivakosky in \cite{Spiv} classified plane valuations in five types according the structure of their dual graphs. In this paper  we are interested in quasi-monomial valuations in the terminology of \cite{FavJon1}, which correspond with those belonging to two of the types in the Spivakovsky' classification: divisorial and irrational valuations. {\it Divisorial valuations} are those whose attached sequence (\ref{Eq_sequencepointblowingups}) is finite while for {\it irrational valuations} the sequence is infinite and there is an index $n$ such that every point $p_j \in \mathcal{C}_\nu$, $j\geq n$, is satellite but there is no point in $\mathcal{C}_\nu$ admitting infinitely many others proximate to it.

The dual graph $\Gamma_\nu$ of a divisorial valuation has the shape in Figure \ref{Fig_dualgraphdivisorial}, where $g$ is a nonnegative integer. When $g >0$ we add labels $st_j$, $0 < j \leq g$, (respectively, $l_j$, $0 \leq j \leq g+1$), to the vertices with degree $3$ (respectively, with degree $1$). Notice that we have set $l_0 = \mathbf{1}$ and it could happen that $st_{g} = l_{g+1}$, being both with degree two.  When $g=0$ we set $l_0 = \mathbf{1} = st_{0}$ and $l_1$ as above. The graph is divided in subgraphs $\Gamma_\nu^j$, $1 \leq j \leq g+1$, where $\Gamma_\nu^1$ contains the vertices (and attached edges) with labels $1 \leq i \leq st_1$, $\Gamma_\nu^j$, $2 \leq j \leq g$, those with labels $st_{j-1} \leq i \leq st_j$ and $\Gamma_\nu^{g+1}$ (usually named the tail of $\Gamma_\nu$) the vertices labelled $i$, $st_g \leq i \leq l_{g+1}$. Notice that when $g=0$, $st_1$ is not defined and $\Gamma_\nu$ coincides with its tail.


\begin{center}
\begin{figure}[h]
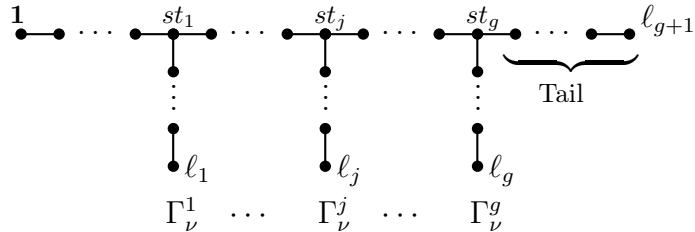


\setlength{\unitlength}{0.5cm}%
\begin{Picture}(0,0)(20,8)
\thicklines


\xLINE(0,6)(1,6)
\Put(0,6){\circle*{0.3}}
\Put(1,6){\circle*{0.3}}
\put(1,6){$\;\;\ldots\;\;$}
\xLINE(3,6)(4,6)
\Put(3,6){\circle*{0.3}}
\Put(4,6){\circle*{0.3}}
\xLINE(4,6)(4,5)
\Put(4,5){\circle*{0.3}}
\Put(3.9,4){$\vdots$}
\xLINE(4,3.5)(4,2.5)
\Put(4,3.5){\circle*{0.3}}
\Put(4,2.5){\circle*{0.3}}
\Put(3.8,1){$\Gamma_{\nu}^1$}

\put(4.3,2.2){$\ell_1$}

\Put(3.7,6.3){\footnotesize $st_1$}
\Put(-0.3,6.3){\footnotesize $\mathbf{1}$}



\xLINE(4,6)(5,6)
\Put(4,6){\circle*{0.3}}
\Put(5,6){\circle*{0.3}}
\Put(5,6){$\;\;\ldots\;\;$}
\xLINE(7,6)(8,6)
\Put(7,6){\circle*{0.3}}
\Put(8,6){\circle*{0.3}}
\xLINE(8,6)(8,5)
\Put(8,5){\circle*{0.3}}
\Put(7.9,4){$\vdots$}
\xLINE(8,3.5)(8,2.5)
\Put(8,3.5){\circle*{0.3}}
\Put(8,2.5){\circle*{0.3}}
\Put(7.8,1){$\Gamma_{\nu}^j$}

\put(8.3,2.2){$\ell_j$}

\Put(7.7,6.3){\footnotesize $st_j$}

\Put(5,1){$\;\;\cdots\;\;$}


\xLINE(8,6)(9,6)
\Put(8,6){\circle*{0.3}}
\Put(9,6){\circle*{0.3}}
\put(9,6){$\;\;\ldots\;\;$}
\xLINE(11,6)(12,6)
\Put(11,6){\circle*{0.3}}
\Put(12,6){\circle*{0.3}}
\xLINE(12,6)(12,5)
\Put(12,5){\circle*{0.3}}
\Put(11.9,4){$\vdots$}
\xLINE(12,3.5)(12,2.5)
\Put(12,3.5){\circle*{0.3}}
\Put(12,2.5){\circle*{0.3}}
\Put(11.8,1){$\Gamma_{\nu}^g$}

\put(12.3,2.2){$\ell_g$}

\Put(11.7,6.3){\footnotesize $st_g$}

\Put(9,1){$\;\;\cdots\;\;$}


\xLINE(12,6)(13,6)
\Put(13,6){\circle*{0.3}}
\Put(13,6){$\;\;\ldots\;\;$}
\Put(15,6){\circle*{0.3}}
\xLINE(15,6)(16,6)
\Put(16,6){\circle*{0.3}}


\Put(12.7,5.5){$\underbrace{\;\;\;\;\;\;\;\;\;\;\;\;\;\;\;}$}
\Put(13.6,4.2){{\footnotesize Tail}}

\Put(16.3,6.2){$\ell_{g+1}$}

\end{Picture}
  \caption{Dual graph of a divisorial valuation}
  \label{Fig_dualgraphdivisorial}
\end{figure}
\end{center}

The dual graph of an irrational valuation is like that of a divisorial valuation with the exception of $\Gamma_\nu^{g+1}$. This subgraph is an infinite graph and their vertices excluding $st_{g}$ and $l_{g+1}$ have degree two.

\subsection{Some invariants}

There exist several invariants usually considered for studying plane valuations. In this paper we will only consider the following two.

\subsubsection{Puiseux exponents}
\label{puis}
Denote by $\mathfrak{m}_i$ the maximal ideal of the local ring $\mathcal{O}_{Z_i,p_{i+1}}$, $i\geq 0$ and write $\nu(\mathfrak{m}_i):=\min\{\nu(x) \,|\, x\in\mathfrak{m}_i\setminus \{0\}\}$. Following \cite{Spiv} (see also \cite{DelGalNun}), the sequence of Puiseux exponents of a divisorial or irrational valuation $\nu$ is an ordered set of real numbers $\{\beta_j'(\nu)\}_{j=0}^{g+1}$ such that $\beta_0'(\nu)=1$ and for $j \in \{1, 2, \ldots, g+1\}$, $\beta_j'(\nu)$ is defined by the continued fraction
\[
\langle a_1^j;a_2^j,\ldots,a_{r_j}^j \rangle,
\]
where $a_k^j$, $1 \leq k < r_j+1$ successively counts the number of vertices in $\Gamma_\nu^j$ with the same value $\nu(\mathfrak{m}_i)$. Notice that when $\nu$ is divisorial $\beta_j'(\nu) \in \mathbb{Q}$ for all $j >0$ and $\beta_{g+1}'(\nu)$ is a positive integer. In case $\nu$ is irrational, then $\beta_j'(\nu) \in \mathbb{Q}$ for all $0 < j <g+1$ but $r_{g+1} = \infty$ and $\beta_{g+1}'(\nu) \in \mathbb{R} \setminus \mathbb{Q}$.

\subsubsection{Maximal contact values}
Let $\nu$ be a divisorial or irrational valuation. For $0\leq j\leq g$, write $\beta_j'(\nu)=q_j/p_j\text{ such that gcd}(q_j,p_j)=1$, define $e_g(\nu)=1$ and when $g >0$, set $e_j(\nu)=\prod_{k=j+1}^g p_k$. Following \cite[Sections 6 and 8]{Spiv} and \cite{DelGalNun}, the sequence of \emph{maximal contact values}  $\{\overline{\beta}_j(\nu)\}_{j=0}^{g+1}$ of $\nu$ is defined as follows.
$\overline{\beta}_0(\nu)$ $: =e_0(\nu)$ and
\begin{equation}\label{Cond_betabarraenbetaprima}
\overline{\beta}_{j+1}(\nu)=e_j(\nu)(\beta_{j+1}'(\nu)-1)+ n_j(\nu)\overline{\beta}_j(\nu), \text{ for }0\leq j\leq g,
\end{equation}
where $n_0(\nu)=1$ and, for $j >0$, $n_j(\nu)=e_{j-1}(\nu)/e_j(\nu)$. Notice that $$e_j(\nu)=\text{gcd}(\overline{\beta}_0(\nu),\ldots,\overline{\beta}_j(\nu)).$$
In addition it holds that the maximal contact values satisfy the following equalities:
\begin{equation}\label{Eq_betabarravaluationcurvettes}
\overline{\beta}_0(\nu)= \nu(\mathfrak{m}) \text{ and } \overline{\beta}_j(\nu) = \nu(\varphi_{l_j}), 0\leq j\leq g+1 \text{ (with }l_0:=1),
\end{equation}
where $\varphi_{l_j}$ is an analytically irreducible element of $R$ whose strict transform on $Z_{l_j}$ is non-singular and transversal to the exceptional divisor $E_{l_j}$. Moreover the set $\{\overline{\beta}_j\}_{j=0}^{g+1}$  generates the semigroup of values of $\nu$, $\nu(R\setminus\{0\})\cup\{0\}$.

Given a divisorial valuation $\nu$, we denote by $\nu^N$ its normalization, that is the equivalent valuation (in the Zariski-Samuel's sense \cite[page 32]{ZarSam})) such that $\nu^N(\mathfrak{m})=1$. In this paper irrational valuations will ever be considered normalized. Then, if $\nu$ is an irrational valuation and $\nu_i^N$ the normalized divisorial valuation defined by the divisor $E_i$ in the sequence (\ref{Eq_sequencepointblowingups}) associated to $\nu$, by  \cite[Theorem 6.1]{Gal}, it holds that
$$
\nu(f)=\lim_{i\to\infty}\nu_i^N(f), \text{for all }f\in K^*.
$$

\section{NPI valuations}\label{Sect_Dos}

In this section we introduce some families of plane valuations which provide interesting geometrical information.

From now on we consider sequences of point blowing-ups as in (\ref{Eq_sequencepointblowingups}) but our starting set $Z_0$ will be either the projective plane $\mathbb{P}^2_k := \mathbb{P}^2$ or the $\delta$th Hirzebruch surface $\mathbb{F}_\delta :=\mathbb{P}\left(\mathcal{O}_{\mathbb{P}^1_k} \oplus \mathcal{O}_{\mathbb{P}^1_k}(-\delta)\right)$, $k$ being an algebraically closed field. The corresponding valuations will be named {\it valuations of $Z_0$} since they are valuations of the field of rational functions of $Z_0$, $K(Z_0)$, centered at $\mathcal{O}_{Z_0,p}$, $p$ being a closed point in $Z_0$.

Recall that $\mathbb{F}_\delta$ is a projective  ruled surface over $k$ equipped with a  projective morphism $pr:\mathbb{F}_\delta\to \mathbb{P}^1_k$. Denote by $F$ a fiber of $pr$ and by $M_0$ a section of $pr$ with  self-intersection $-\delta$; when $\delta >0$, $M_0$  is unique  and it is called the \emph{special section}. Set $M$ any linearly equivalent to $\delta F + M_0$ section of $pr$. By convention, those points in $\mathbb{F}_\delta$, $\delta >0$, that belong to $M_0$ are called \emph{special points}. Otherwise we name them \emph{general points}.
\begin{definition}
\label{def31}
{\rm Let $p$ be a point in $\mathbb{F}_\delta$ and $\nu$ a divisorial valuation of $\mathbb{F}_\delta$. The valuation $\nu$ is said to be \emph{special} (with respect to $\mathbb{F}_\delta$ and $p$) if one of the following condition is satisfied:
\begin{itemize}
\item[(1)] $\delta = 0$.
\item[(2)] $\delta >0$ and $p$ is a special point.
\item[(3)] $\delta >0$, $p$ is a general point and there is no integral curve in the complete linear system $|M|$, given by the section $M$, whose strict transform on the surface defined by $\nu$ has negative self-intersection.
\end{itemize}
If $\nu$ does not satisfy any of the above conditions, it will be called \emph{non-special}.

\medskip

An irrational valuation of $\mathbb{F}_\delta$ is called to be {\it special} (respectively, {\it non-special}) whenever $\nu(f)$, $f \in K(\mathbb{F}_\delta)$ can be computed as a limit of normalized special (respectively, non-special) divisorial valuations $\nu_i^N(f)$.}
\end{definition}

Next we are going to introduce the concept of non-positive at infinity valuation of $Z_0$. If $Z_0=\mathbb{P}^2$, denote by $L$ a projective line (the line at infinity) containing $p$. When $Z_0=\mathbb{F}_\delta$, stand $F_1$ for the fiber of $pr$ going through $p$ and, if $\nu$ is non-special, set $M_1$ the section of $pr$ whose strict transform on the surface defined by $\nu$ (if $\nu$ is divisorial) or by $\nu_i^N$, $i \gg 0$ (otherwise)  has negative self-intersection.

\begin{definition}
\label{def32}
{\rm
Let $\nu$ be a divisorial, or irrational, valuation of $Z_0$.

When $Z_0=\mathbb{P}^2,$ we say that $\nu$ is \emph{non-positive at infinity (NPI)}  if $\nu(f)\leq 0$ for all $f\in\mathcal{O}_{\mathbb{P}^2}(\mathbb{P}^2\setminus L)$.

In case $Z_0=\mathbb{F}_\delta$, a special (respectively, non-special) valuation $\nu$ is called \emph{non-positive at infinity (NPI)} if $\nu(f)\leq 0$ for all $f\in\mathcal{O}_{\mathbb{F}_\delta}(\mathbb{F}_\delta\setminus(F_1\cup M_0))$ (respectively, for all
$f\in\mathcal{O}_{\mathbb{F}_\delta}(\mathbb{F}_\delta\setminus(F_1\cup M_1))$).}
\end{definition}

NPI divisorial valuations enjoy nice geometrical properties. Next we briefly describe some of them.

Let $\nu_n$ be a divisorial valuation of $Z_0$ whose sequence (\ref{Eq_sequencepointblowingups}) ends at the surface $Z_n=Z$. Denote Pic$(Z)$ the Picard group of the surface $Z$ and set NE$(Z)$ the cone of curves of that surface regarded as the convex cone of Pic$_{\mathbb{Q}}(Z) = \mathrm{Pic} (Z) \otimes \mathbb{Q}$ generated by the classes of effective divisors on $Z$.

Denote by $\varphi_C$ the germ at $p$ of a curve $C$ on $Z_0$. Moreover set $\varphi_i$, $1\leq i\leq n,$ an analytically irreducible germ of curve at $p$ whose strict transform on $Z_i$ is transversal to $E_i$ at a non-singular point of the exceptional locus. We will also stand $(\phi,\varphi)_p$ for the intersection multiplicity at $p$ of two germs of curve at $p$, $\phi$ and $\varphi$, on $Z_0$.

The following theorem states the above mentioned geometrical properties. Proofs can be found in \cite{GalMon,GalMonMor}.

\begin{theorem}
\label{teo33}
Let $\nu_n$ be a divisorial plane valuation of $Z_0$. Set $Z$ the surface defined by $\nu_n$. Then
\begin{itemize}
\item[(a)] If $Z_0=\mathbb{P}^2$, then the following conditions are equivalent:
\begin{itemize}
\item[(a.1)] The valuation $\nu_n$ is non-positive at infinity.
\item[(a.2)] The divisor $(\varphi_L,\varphi_n)_p E_0^* + \sum_{i=1}^n\nu(\mathfrak{m}_i)E_i^*$ is nef.
\item[(a.3)] The cone of curves $NE(Z)$ is generated by the classes in $\mathrm{Pic}_{\mathbb{Q}}(Z)$ of the divisors $\tilde{L},E_1,E_2,\ldots, E_n$.

\end{itemize}
\medskip
\item [(b)] If $Z_0=\mathbb{F}_\delta$, $\delta\geq 0$ and $\nu_n$ is special, the following statements are equivalent:
\begin{itemize}
\item[(b.1)] The valuation $\nu_n$ is non-positive at infinity.
\item[(b.2)] The divisor $(\varphi_{M_0},\varphi_n)_p F^* + (\varphi_{F_1},\varphi_n)_p M^* + \sum_{i=1}^n\nu(\mathfrak{m}_i)E_i^*$ is nef.
\item[(b.3)] The cone of curves $NE(Z)$ is generated by the classes in $\mathrm{Pic}_{\mathbb{Q}}(Z)$ of the divisors $\tilde{F}_1,\tilde{M}_0,E_1,E_2,\ldots, E_n$.
\end{itemize}
\medskip
\item[(c)] If $Z_0=\mathbb{F}_\delta$, $\delta\geq 1$, and $\nu_n$ is non-special, the following conditions are equivalent:
\begin{itemize}
\item[(c.1)] The valuation $\nu_n$ is non-positive at infinity.
\item[(c.2)] The divisor $(\varphi_{M_1},\varphi_n)_p F^* + (\varphi_{F_1},\varphi_n)_p M^* + \sum_{i=1}^n\nu(\mathfrak{m}_i)E_i^*$ is nef.
\item[(c.3)] The cone of curves $NE(Z)$ is generated by the classes in $\mathrm{Pic}_{\mathbb{Q}}(Z)$ of the divisors $\tilde{F}_1,\tilde{M}_0,\tilde{M}_1,E_1,E_2,\ldots, E_n$.
\end{itemize}
\end{itemize}
\end{theorem}

\section{Discrete equivalence of NPI valuations}\label{Sect_Tres}

Dual graphs of plane valuations are key objects for classifying them. NPI valuations are a subclass of plane valuations encoding  interesting global geometric information. We think it is interesting to decide which are the dual graphs of plane valuations admitting NPI valuations. Being more explicit, given a dual graph $\Gamma$  corresponding to a divisorial or irrational valuation, we desire to know if there is some NPI valuation (and which is its type) whose dual graph is $\Gamma$.

From our description in Subsection \ref{puis}, it is clear that the sequence of Puiseux exponents of a valuation is an equivalent datum to its dual graph which is simpler to manipulate. This gives a meaning to the following definition.

\begin{definition}
\label{Def_dicrete_class}
{\rm
Let $\nu$ be a divisorial or irrational valuation and $\{\beta_j'(\nu)\}_{j=0}^{g+1}$ its sequence of Puiseux exponents. The \emph{discrete class} of $\nu$ is the tuple $$\mathbf{t}_\nu=(g,\beta_0'(\nu),\beta_1'(\nu),\ldots, \beta_{g+1}'(\nu)).$$

Two valuations as above having the same discrete class are called {\it discretely equivalent}.}
\end{definition}
The following result is immediate.
\begin{proposition}
Let $\nu$ and $\nu'$ be two divisorial or irrational valuations. Then the dual graphs of $\nu$ and $\nu'$ coincide if and only if $\nu$ and $\nu'$ have the same discrete class.
\end{proposition}

Now we give a lemma we will use for characterizing the discrete classes of NPI valuations.

\begin{lemma}
\label{Lemma_Equality_betabarragmas1}
Keep the notation as in Section \ref{Sect_Dos}.
Let $\nu$ be a divisorial or irrational valuation of $Z_0$ and $\mathbf{t}_\nu=(g,\beta_0'(\nu),$ $\beta_1'(\nu),\ldots,$ $\beta_{g+1}'(\nu))$ its discrete class. Then the last maximal contact value of $\nu$ satisfies
$$
\overline{\beta}_{g+1}(\nu)=\sum_{j=1}^g e_j(\nu)^2(\beta_{j+1}'(\nu)-1) + \overline{\beta}_0(\nu)\overline{\beta}_1(\nu),
$$
where, by convention, $\sum_{j=1}^g e_j(\nu)^2(\beta_{j+1}'(\nu)-1)=0$ when $g=0$.
\end{lemma}
\begin{proof}
The case $g=0$ is obvious. Assume $g >0$, by Equality  \eqref{Cond_betabarraenbetaprima} and since $e_g(\nu)=\text{gcd}(\overline{\beta}_0(\nu),\overline{\beta}_1(\nu),\ldots,\overline{\beta}_g(\nu))=1$, it holds that
$$
\overline{\beta}_{g+1}(\nu) = e_g(\nu)^2(\beta_{g+1}'(\nu)-1)+e_{g-1}(\nu)\overline{\beta}_g(\nu).
$$
Now, again by \eqref{Cond_betabarraenbetaprima}, we obtain that
$$
e_j(\nu)\overline{\beta}_{j+1}(\nu)=e_j(\nu)^2(\beta_{j+1}'(\nu)-1) + e_{j-1}(\nu)\overline{\beta}_{j}(\nu),
$$
for $1\leq j \leq g-1$. This concludes the proof after noticing that $e_0 (\nu) =  \overline{\beta}_{0} (\nu)$.
\end{proof}

The discrete classes of divisorial valuations are tuples $\mathbf{t}= (g,\beta_0',\beta_1',\ldots,\beta_{g+1}')$, where $g$ is a nonnegative integer,  $\beta_{g+1}'$ a positive integer, $\beta'_0=1$  and (when $g \neq 0$) $\beta_j' \in \mathbb{Q}_{>0} \setminus \mathbb{Z}$, $1 \leq j\leq g$, $\mathbb{Q}_{>0}$ being the positive rational numbers, which, for simplicity, we write as
$\beta_j'=q_j/p_j$ where gcd$(p_j,q_j)=1$. Furthermore for each $\mathbf{t}$, we set $e_g=1$ and in case $g \neq 0$, $e_j=\prod_{k=j+1}^g p_k$, $0\leq j \leq g-1,$ and $w_j=e_j/e_0$, for  $0\leq j\leq g$. When $\nu$ is an irrational valuation, its discrete class is defined analogously but $\beta_{g+1}'$ is an irrational number.

The set of tuples $\mathbf{t}$ as above is denoted by $\mathbb{T}$. Our next result uses the sum $\sum_{j=1}^g w_j^2(\beta_{j+1}'-1)$ which, by convention, is defined to be zero whenever $g=0$.

\begin{theorem}
\label{Th_CondNumeric_divisorial}
Fix a class $\mathbf{t} \in \mathbb{T}$. Then
\begin{itemize}
\item[(a)] The tuple $\mathbf{t}$ is the discrete class of some NPI divisorial or irrational valuation of $\mathbb{P}^2$ if, and only if, the following equality holds:
$$
\beta_1'(\beta_1'-1)\geq \sum_{j=1}^g w_j^2(\beta_{j+1}'-1).
$$

\item[(b)] The tuple $\mathbf{t}$ is the discrete class of some NPI special divisorial or irrational valuation of $\mathbb{F}_\delta$ if, and only if, $\delta$ is a non-negative integer and
$$
\beta_1'(\delta\beta_1'+1)\geq \sum_{j=1}^g w_j^2(\beta_{j+1}'-1).
$$

\item[(c)]The tuple $\mathbf{t}$ is the discrete class of some  NPI non-special divisorial or irrational valuation of $\mathbb{F}_\delta$ if, and only if, $\delta$ is a positive integer and
$$
\beta_1'-\delta\geq \sum_{j=1}^g w_j^2(\beta_{j+1}'-1).
$$

\end{itemize}
\end{theorem}

\begin{proof}
It suffices to prove the result for the set of discrete classes of divisorial valuations. The result for irrational valuations follows from \cite[Theorem 6.1]{Gal}.

We assume that $g \neq 0$ because the case $g=0$ can be deduced reasoning analogously. Set $\mathbf{t}= (g,\beta_0',\beta_1',\ldots,\beta_{g+1}')$ a discrete class where $\beta_{g+1}'$ is an integer number. Let $\overline{\beta}_0, \overline{\beta}_1, \overline{\beta}_{g+1}$ be the values computed from the components of $\mathbf{t}$ following the the formulas given in Equality \eqref{Cond_betabarraenbetaprima} and Lemma  \ref{Lemma_Equality_betabarragmas1}.

We are going to prove that there exists an NPI divisorial valuation of $\mathbb{P}^2$ whose discrete class is $\mathbf{t}$ if and only if $\overline{\beta}_1^2\geq\overline{\beta}_{g+1}.$ This will prove Statement (a)  because there is an equivalence between the above inequality and that of Statement (a), which follows from the equality  $\beta'_1=\overline{\beta}_1/\overline{\beta}_0$ and Lemma \ref{Lemma_Equality_betabarragmas1}.

Assume first that $\nu$ is an NPI divisorial valuation of $\mathbb{P}^2$ whose discrete class is $\mathbf{t}$. With the notation used before Theorem \ref{teo33}, by \cite[Theorem 1]{GalMon}, it holds that
$$
(\varphi_L,\varphi_n)_p^2\geq \overline{\beta}_{g+1}.
$$
Now $(\varphi_L,\varphi_n)_p=\overline{\beta}_1$ holds
when the strict transform of $L$ goes through all initial free points in $\mathcal{C}_\nu$. Otherwise, $(\varphi_L,\varphi_n)_p=s_L\overline{\beta}_0,$ where $1\leq s_L\leq \lfloor\overline{\beta}_1/\overline{\beta}_0\rfloor$. As a consequence one obtains that $\overline{\beta}_1^2\geq \overline{\beta}_{g+1}$ and this implication is proved. Conversely, if we take a discrete class $\mathbf{t} \in \mathbb{T}  $  (where $\beta_{g+1}'$ is a positive integer number) such that $\overline{\beta}_1^2\geq \overline{\beta}_{g+1}$ it is sufficient to consider any divisorial valuation of $\mathbb{P}^2$ whose discrete class is $\mathbf{t}$ and whose first free points in $\mathcal{C}_\nu$ are on the  projective line $L$ and its strict transforms. Therefore the equality $(\varphi_L,\varphi_n)_p=\overline{\beta}_{1}$ is satisfied and the proof is concluded by \cite[Theorem 1]{GalMon}.

As in the proof of Item (a), for proving the equivalence stated in Item (b) it suffices to replace the inequality given there with $$2\overline{\beta}_1\overline{\beta}_0+\delta\overline{\beta}_1^2\geq \overline{\beta}_{g+1}.$$
 Now assume that $\mathbf{t}$ is the discrete class of a special divisorial valuation of $\mathbb{F}_\delta$. With notation as in Theorem \ref{teo33}, by \cite[Theorem 3.6]{GalMonMor}, the inequality
\begin{equation}
\label{Delta}
2(\varphi_{F_1},\varphi_n)_p(\varphi_{M_0},\varphi_n)_p+ \delta(\varphi_{F_1},\varphi_n)_p^2\geq \overline{\beta}_{g+1}
\end{equation}
holds. Now $(\varphi_{F_1},\varphi_n)_p$ equals either $\overline{\beta}_1$  if the strict transform of $F_1$ goes through all initial free points in $\mathcal{C}_\nu$ or $s_{F_1}\overline{\beta}_0$, where $1\leq s_{F_1}\leq \lfloor\overline{\beta}_1/\overline{\beta}_0\rfloor,$ otherwise. The section $M_0$ has a similar behaviour and then $(\varphi_{M_0},\varphi_n)_p$ equals either $\overline{\beta}_1$ or $s_{M_0}\overline{\beta}_0$, where $0\leq s_{M_0}\leq \lfloor\overline{\beta}_1/\overline{\beta}_0\rfloor$. This proves that $2\overline{\beta}_1\overline{\beta}_0+\delta\overline{\beta}_1^2\geq \overline{\beta}_{g+1}$ because $F_1$ and $M_0$ go both through $p$ but at most one of their strict transforms passes through $p_2$. Conversely,  let $\mathbf{t} \in \mathbb{T}$ be ($\beta_{g+1}' \in \mathbb{Z}_{>0}$) such that the inequality $2\overline{\beta}_1\overline{\beta}_0+\delta\overline{\beta}_1^2\geq \overline{\beta}_{g+1}$ is satisfied, then considering the special divisorial valuation $\nu$ of $\mathbb{F}_\delta$ with discrete class $\mathbf{t}$ whose first free points in $\mathcal{C}_\nu$ coincide with those through which $F_1$ goes, by \cite[Theorem 3.6]{GalMonMor} one gets an NPI special divisorial valuation on $\mathbb{F}_\delta$ with discrete class $\mathbf{t}$.

Finally reasoning as above one can give a proof of Statement (c) which is supported in the following two facts. First, by \cite[Theorem 4.8]{GalMonMor} the inequality in the statement is equivalent to the following one
$$
2(\varphi_{F_1},\varphi_n)_p(\varphi_{M_1},\varphi_n)_p- \delta(\varphi_{F_1},\varphi_n)_p^2\geq \overline{\beta}_{g+1}.
$$
Second, the fiber $F_1$ and the section $M_1$ satisfy that $(\varphi_{F_1},\varphi_n)_p=\overline{\beta}_0$ and $(\varphi_{M_1},\varphi_n)_p$ is equal either to $\overline{\beta}_1$ or to $s_{M_1}\overline{\beta}_0,$ where $\delta +1\leq s_{M_1}\leq \lfloor\overline{\beta}_1/\overline{\beta}_0\rfloor$. This concludes the proof.

\end{proof}

Plane valuations are algebraic objects which are not linked to specific surfaces. Our next result explains the relation among the dual graphs of the different types of NPI valuations. Denote by $\mathfrak{C}_{\mathbb{P}^2}$ (res\-pectively, $\mathfrak{C}_{\mathbb{F}_\delta}^1$,  $\mathfrak{C}_{\mathbb{F}_\delta}^2$) the set of discrete classes in $\mathbb{T}$ admitting some NPI (respectively, NPI special, NPI non-special)  divisorial or irrational valuation of $\mathbb{P}^2$  (respectively, $\mathbb{F}_\delta$).

\begin{theorem}
\label{Thm_Comparacion_grafos}
The following inclusions among the above introduced sets of discrete classes hold.
\begin{itemize}
\item[(a)] $\mathfrak{C}_{\mathbb{P}^2} \subseteq \mathfrak{C}_{\mathbb{F}_\delta}^1$ for all $\delta >0$.
\item[(b)] Let $\mathfrak{C}_{\mathbb{P}^2}^{\leq 2}$ (respectively, $\mathfrak{C}_{\mathbb{P}^2}^{\geq 2}$) be the set of discrete classes in $\mathfrak{C}_{\mathbb{P}^2}$ whose coordinate $\beta_1'$ satisfies $\beta_1'\leq 2$ (respectively, $\beta_1'\geq 2$). Define analogously $\mathfrak{C}_{\mathbb{F}_0}^{1,\leq 2}$ and $\mathfrak{C}_{\mathbb{F}_0}^{1,\geq 2}$. Then
    $$
    \mathfrak{C}_{\mathbb{P}^2}^{\leq 2} \subseteq \mathfrak{C}_{\mathbb{F}_0}^{1,\leq 2} \;\; \mbox{and } \; \;\mathfrak{C}_{\mathbb{F}_0}^{1,\geq 2} \subseteq \mathfrak{C}_{\mathbb{P}^2}^{\geq 2}.
    $$
    \item[(c)] $\mathfrak{C}_{\mathbb{F}_\delta}^2\subseteq\mathfrak{C}_{\mathbb{P}^2}$ for all $\delta \geq 0$.

\end{itemize}
\end{theorem}
\begin{proof}
By the proof of Theorem \ref{Th_CondNumeric_divisorial}, for any $\delta >0$ it holds that $\mathbf{t} \in\mathfrak{C}_{\mathbb{P}^2}$ (respectively, $\mathbf{t}\in\mathfrak{C}^1_{\mathbb{F}_\delta}$) if and only if $\beta_1'^2\geq \overline{\beta}_{g+1}/\overline{\beta}_0^2$ (respectively, $2\beta_1'+\delta\beta_1'^2\geq \overline{\beta}_{g+1}/\overline{\beta}_0^2$). The fact that assuming $\beta_1'^2\geq \overline{\beta}_{g+1}/\overline{\beta}_0^2$, the inequality $2\beta_1'+\delta\beta_1'^2\geq \beta_1'^2$ is true if $\delta$ is a positive integer and also when $\delta=0$ and $\beta_1'(\nu)\leq 2$ proves (a) and the first inclusion in (b).

 The second inclusion in (b) follows from an analogous argument which leads to the opposite inequality.

To conclude, again by Theorem \ref{Th_CondNumeric_divisorial}, $\mathbf{t}$ belongs to $\mathfrak{C}_{\mathbb{F}_\delta}^2$, $\delta >0$, if and only if $2\beta_1'-\delta\geq \overline{\beta}_{g+1}/\overline{\beta}_{0}^2$. Then, the fact that $\beta_1'^2\geq 2\beta_1'-\delta$ completes the proof.
\end{proof}

\begin{remark}
{\rm
The inclusions proved in the above theorem are strict. The following examples show it.

(a) Consider the discrete class $\mathbf{t}=(2,1,4/3,17/3,1)$. Then $e_0=9,e_1=3$ and $e_2=1$. In addition,
$$
\beta_1'(\beta_1'-1)=\dfrac{12}{27} \text{ and } \sum_{i=1}^g w_i^2(\beta_{i+1}'(\nu)-1) = \dfrac{14}{27},
$$
which proves that $\mathbf{t} \not \in \mathfrak{C}_{\mathbb{P}^2}$. Now,
$$\beta_1'(\delta\beta_1'+1)= \frac{48\delta + 36}{27},$$
shows that $ \mathbf{t }\in\mathfrak{C}_{\mathbb{F}_\delta}^1$ for all non-negative integer $\delta$. Therefore the inclusion in Theorem \ref{Thm_Comparacion_grafos} (a) is strict.
\medskip

(b) Assume $\delta=0$ and set $\mathbf{t}=(3,1,7/3,43/2,14/3,1)$. Then $e_0=18,e_1=6,e_2=3$ and $e_3=1$ and one has
$$
\beta_1'=\dfrac{252}{108}, \;\;  \sum_{j=1}^g w_j^2(\beta_{j+1}'-1) = \dfrac{257}{108} \text{ and } \beta_1'(\beta_1'-1)=\dfrac{336}{108},
$$
which proves that $\mathbf{t} \in \mathfrak{C}_{\mathbb{P}^2}^{\geq 2}$ but $\mathbf{t}   \not \in \mathfrak{C}_{\mathbb{F}_0}^{1,\geq 2}$.
\medskip

(c)  Finally let $\mathbf{t}=(2,1,5/2,57/5,\Phi)$, where $\Phi$ denotes the golden ratio. One has $e_0=10,e_1=5$ and $e_2=1$. Then
$$
\beta_1'-\delta=\dfrac{5}{2} -\delta \text{ and } \sum_{j=1}^g w_j^2(\beta_{j+1}'-1)= \dfrac{52}{20}+ \dfrac{1}{100}(\Phi-1).
$$
This proves that $\mathbf{t} \not \in \mathfrak{C}_{\mathbb{F}_\delta}^2$ for any positive integer $\delta$. However $\mathbf{t}  \in \mathfrak{C}_{\mathbb{P}^2}$  because $\beta_1'(\beta_1'-1)= 75/20$.
}
\end{remark}
\section{Generating dual graphs of NPI valuations}\label{Sect_Cuatro}

The purpose of this section is to give an algorithmic procedure to provide the discrete classes (i.e., dual graphs) which admit NPI divisorial or irrational valuations. We assume that $g >0$ since any discrete class $(0,1, \beta_1')$ admits an NPI  (respectively, NPI special) valuation of $\mathbb{P}^2$ (respectively, $\mathbb{F}_\delta$) and by Theorem \ref{Th_CondNumeric_divisorial} the discrete classes  $(0,1, \beta_1' \geq \delta)$ are those that admit NPI non-special valuations of $\mathbb{F}_\delta$.

The {\it input} of our algorithm is a discrete class
\begin{equation}
\label{input1}
\mathbf{t}(\nu_I) :=(g,\beta_0'(\nu_I),\beta_1'(\nu_I),\ldots, \beta_{g+1}'(\nu_I)=1)
\end{equation}
belonging to one of the following sets: $\mathfrak{C}_{\mathbb{P}^2}$, $\mathfrak{C}_{\mathbb{F}_\delta}^1$ or  $\mathfrak{C}_{\mathbb{F}_\delta}^2$ generically denoted by $\mathfrak{C}$. It will give two outputs:

{\it Output 1,} which is a new discrete class in the same set $\mathfrak{C}$ of the input   with the shape:
\begin{multline}
\label{output1}
 \mathbf{t}(\nu_{O_1}) : =\big(g+1,\beta_0'(\nu_{O_1})= \beta_0'(\nu_I), \beta_1'(\nu_{O_1})=\beta_1'(\nu_I),\\
 \ldots,\beta_g'(\nu_{O_1})=\beta_g'(\nu_{I}), \beta_{g+1}'(\nu_{O_1}), \beta_{g+2}'(\nu_{O_1})=1\big).
\end{multline}

{\it Output 2.} It is in fact a double output. On the one hand a  new discrete class $\mathbf{t}(\nu_{O_2^1})$  in the same set $\mathfrak{C}$ of the input with the shape:
\begin{multline}
\label{output2}
 \mathbf{t}(\nu_{O_2^1}) :=\big(g,\beta_0'(\nu_{O_2^1})= \beta_0'(\nu_{O_1}), \beta_1'(\nu_{O_2^1})= \beta_1'(\nu_{O_1}),\\
 \ldots,\beta_g'(\nu_{O_2^1})= \beta_g'(\nu_{O_1}), \beta_{g+1}'(\nu_{O_2^1})\big),
\end{multline}
where $\beta_{g+1}'(\nu_{O_2^1}) \in \mathbb{R}_{>0} \setminus \mathbb{Q}_{>0}$.

And, on the other hand, a new discrete class $\mathbf{t}(\nu_{O_2^2})$ in the same set
$\mathfrak{C}$ of the input with the shape:

\begin{multline}
\label{output3}
 \mathbf{t}(\nu_{O_2^2}) :=\big(g+1, \beta_0'(\nu_{O_2^2})= \beta_0'(\nu_{O_1}) \beta_1'(\nu_{O_2^2})= \beta_1'(\nu_{O_1}),\\ \ldots, \beta_{g+1}'(\nu_{O_2^2})= \beta_{g+1}'(\nu_{O_1}), \beta_{g+2}'(\nu_{O_1^2})\big),
\end{multline}
where $\beta_{g+2}'(\nu_{O_1^2})$ is a positive integer different from 1.

The outputs are not unique but we can get infinitely many tuples as Output 1 and also in the first Output 2.

It is clear that the algorithm we are going to describe starts with a tuple corresponding to an NPI  divisorial valuation defined by a satellite divisor whose dual graph $\Gamma$ has $g$ subgraphs $\Gamma^j$. It provides dual graphs corresponding to the same type of NPI valuations which keep the subgraphs and add either a new subgraph  $\Gamma^{g+1}$ corresponding to a divisorial valuation defined by a satellite divisor or an irrational valuation. Even more, it also gives  dual graphs corresponding to NPI divisorial valuations defined by free divisors keeping the $g$ subgraphs $\Gamma^j$ and adding two new subgraphs $\Gamma^{g+1}$ and (a tail) $\Gamma^{g+2}$.

Notice that, starting with suitable tuples $(1,\beta_0'=1,\beta_1',1)$, we are able to provide the dual graph of any NPI valuation of any desired type with $g$ as large as we want.

Let us show the simple steps of our algorithm. With the the above notation, define
$$
q(\mathbf{t}(\nu)):=\left\lbrace
\begin{array}{cl}
\beta_1'(\nu)(\beta_1'(\nu)-1), & \text{ if } Z_0=\mathbb{P}^2, \\[2mm]
\beta_1'(\nu)(\delta\beta_1'(\nu)+1), & \text{ if }Z_0=\mathbb{F}_\delta \text{ and $\nu$ is special}, \\[2mm]
\beta_1'(\nu)- \delta, &\text{ if }Z_0=\mathbb{F}_\delta \text{ and $\nu$ is non-special,}
\end{array}\right.
$$
where $\beta_1'(\nu)$ denotes the third coordinate of $\mathbf{t}(\nu)$.  Our algorithm will have the following\\
{\bf Input:} A discrete class as in (\ref{input1}) which satisfies the inequality
\begin{equation}
\label{entrada}
q\left(\mathbf{t}(\nu_I)\right)>\sum_{j=1}^{g-1} e_j\left(\nu_I\right)^2\left(\beta_{j+1}'(\nu_I)-1\right),
\end{equation}
condition we have to impose because we start with an NPI divisorial valuation and we need the strict inequality to have some degree of freedom to add Puiseux exponents.\\
{\bf Output 1:} It will be a tuple as in (\ref{output1}) obtained as follows. Let $\beta_{g+1}'(\nu_{O_1})=q_{g+1}/p_{g+1}$, where $p_{g+1}$ and $q_{g+1}$ are positive integers such that $q_{g+1}>p_{g+1}$ and gcd$(q_{g+1},p_{g+1})=1.$ Then our output must satisfy
\[
\beta_j'(\nu_{O_1})=\beta_j'(\nu_I) \text{ and } e_j(\nu_{{O_1}}^N)=e_j(\nu_I^N), \text{ for }0\leq j \leq g,
\]
$\overline{\beta}_{0}(\nu_{O_1})=p_{g+1}\overline{\beta}_{0}(\nu_I)$ and then one obtains that
$$
e_{g+1}(\nu_{O_1}^N)=\dfrac{1}{\overline{\beta}_{0}(\nu_{O_1})}=\dfrac{1}{p_{g+1}\overline{\beta}_{0}(\nu_I)}=\dfrac{e_g(\nu_I^N)}{p_{g+1}}.
$$
Since we require that $\mathbf{t}(\nu_{O_1})$ corresponds to an NPI valuation, {\it for obtaing our Output 1 it suffices to consider any pair $p_{g+1}$ and $q_{g+1}$ defining $\beta_{g+1}'(\nu_{O_1})$ which satisfies the following inequality:}
\begin{equation}\label{Cond_AmpliarGrafoDual2}
\dfrac{q(\mathbf{t}(\nu_I)) - \sum_{j=1}^{g-1} e_j(\nu_I^N)^2(\beta_{j+1}'(\nu_I)-1)}{e_g(\nu_I^N)^2} + 1 \geq \beta_{g+1}'(\nu_{O_1})= \dfrac{q_{g+1}}{p_{g+1}}>1.
\end{equation}
Notice that the algorithm has to make a choice because one has infinitely many options.\\
{\bf Output 2:} For obtaining $\mathbf{t}(\nu_{O_2^1})$ as described in (\ref{output2}), we only need to look for an irrational number $\beta_{g+1}'(\nu_{O_2^1})$ that satisfies Inequality (\ref{Cond_AmpliarGrafoDual2}) when it reemplaces the rational number $\beta_{g+1}'(\nu_{O_1})$. Notice that, again, we have to make a choice.

Finally to get an output $\mathbf{t}(\nu_{O_2^2})$ as described in (\ref{output3}), it suffices to look for positive integer $\beta_{g+2}'(\nu_{O_2^2}) > 1$ such that
\begin{equation}
\label{elgmas2}
\dfrac{q(\mathbf{t}(\nu_{O_2^2})) - \sum_{j=1}^{g} e_j(\nu_{O_1}^N)^2(\beta_{j+1}'(\nu_{O_1})-1)}{e_{g+1}(\nu_{O_1}^N)^2} \geq \beta_{g+2}' (\nu_{O_2^2}) -1.
\end{equation}
The maximum number of free points  in the tail of the dual graph of $\nu_{O_2^2}$ would be the biggest non-negative integer $\beta_{g+2}'-1$ which satisfies the last inequality.

\medskip

We conclude this paper with two examples that show how our algorithm works.

\begin{example}
{\rm
Assume that $Z_0=\mathbb{P}^2$ and that our input is $\mathbf{t}(\nu_I)= (2,1,5/2,7/5,1)$. Then $\{e_i(\nu_I)\}_{i=0}^2=\{10,5,1\}$. Our input is correct because
$$
\dfrac{15}{4}=\dfrac{5}{2}\left(\dfrac{5}{2}-1\right)>\left(\dfrac{1}{2}\right)^2\left(\dfrac{7}{5}-1\right)=\dfrac{1}{10}
$$
and Inequality (\ref{entrada}) is satisfied.

Now taking into account Inequality (\ref{Cond_AmpliarGrafoDual2}), suitable values $\beta_3'(\nu_{O_1})$ for our purposes are those satisfying
$$
366=\dfrac{\beta_1'(\nu_I)(\beta_1'(\nu_I)-1) - e_1(\nu_I^N)^2(\beta_{2}'(\nu_I)-1)}{e_2(\nu_I^N)^2} + 1 \geq \beta_{3}'(\nu_{O_1}) > 1.
$$

If, for instance, we choose $\mathbf{t}(\nu_{O_1})= (3,1,5/2,7/5,8/3,1)$ we obtain a valid Output 1. We can continue our algorithm for finding values $\beta_4'(\nu_{O_2^2})$ whose according (\ref{elgmas2}) must satisfy $3270\geq \beta_4'(\nu_{O_2^2})-1\geq 0$. So a possible Output 2 would be $\mathbf{t}(\nu_{O_2^2})= (3,1,5/2,7/5,8/3,3200)$.
}
\end{example}

\begin{example}
{\rm
Suppose now that $Z_0=\mathbb{F}_2$ and $\mathbf{t}(\nu_I)=(3,1,5/3,12/5,5/2,1)$. This is a suitable input since $\{e_i(\nu_I)\}_{i=0}^2=\{30,10,2,1\}$,
$$
\dfrac{65}{9}=\dfrac{5}{3}\left(2\dfrac{5}{3}+1\right)>\left(\dfrac{1}{3}\right)^2\left(\dfrac{12}{5}-1\right)+\left(\dfrac{1}{15}\right)^2\left(\dfrac{5}{2}-1\right)=\dfrac{73}{450}
$$
and, then, Inequality (\ref{entrada}) holds. Thus, if we desire an output $\mathbf{t}(\nu_{O_2^1})$, we need to look for irrational numbers such that they satisfy Inequality \eqref{Cond_AmpliarGrafoDual2}. That is values $\beta_4'(\nu_{O_2^1})$ such that $6355\geq\beta_4'(\nu_{O_2^1}) >0.$ Then, $\mathbf{t}(\nu_{O_2^1})=(3,1,5/3,12/5,5/2,\pi)$ is a valid output.
}
\end{example}


\bibliographystyle{plain}
\bibliography{BIBLIO}

\end{document}